\documentclass[12pt]{article}

\usepackage{amssymb,amsthm,amsmath}
\usepackage[utf8]{inputenc}

\newcommand{\R}{\mathbb{R}}
\newcommand{\C}{\mathbb{C}}
\newcommand{\1}{\textbf{1}}
\newcommand{\dd}{\mathrm{d}}

\newcommand{\po}[2]{\frac{\textrm{d} #1}{\textrm{d} #2}}
\newcommand{\call}[4]{\int_{#1}^{#2} {#3} \; \textrm{d} {#4}}
\newcommand{\fun}[3]{#1\colon #2 \longrightarrow #3}

\newtheorem{proposition}{Proposition}
\newtheorem{thm}{Theorem}
\newtheorem{lemma}{Lemma}
\newtheorem{cor}{Corollary}

\theoremstyle{definition}
\newtheorem*{def*}{Definition}

\theoremstyle{remark}
\newtheorem*{rem*}{Remark}

\title{S-inequality for certain product measures}

\author{Piotr Nayar \thanks{Research partially supported by NCN Grant no. 2011/01/N/ST1/01839.}, Tomasz Tkocz \thanks{Research partially supported by NCN Grant no. 2011/01/N/ST1/05960.} }
\date{}

\begin{document}

\maketitle

\begin{abstract}
In this paper we prove the S-inequality for certain product probability measures and ideals in $\R^n$. As a result, for the Weibull and Gamma product distributions we derive concentration of measure type estimates as well as optimal comparison of moments.
\end{abstract}

\noindent {\bf 2010 Mathematics Subject Classification.} Primary 60G15; Secondary 60E15.

\noindent {\bf Key words and phrases.} S-inequality, Dilation, Exponential distribution, Weibull distribution, Gamma distribution, Concentration of measure, Comparison of moments.

\section{Introduction}\label{sec:intro}

The standard Gaussian measure $\gamma_n$ on $\R^n$ has been thoroughly studied in a context of dilations of convex and symmetric sets  (see \cite{CFM, LO1}). For example, it is known that for such a set $K$ in $\R^n$ we have the estimate
\[
	\gamma_n(tK) \geq \gamma_n(tP), \qquad t \geq 1,
\]
where the set $P = \{ x \in \R^n, \ |x_1| \leq p\}$ is a strip chosen so that $\gamma_n(P) = \gamma_n(K)$. This result is due to R. Lata\l a and K. Oleszkiewicz \cite{LO1} and it is called the S-inequality. 
A natural task is to find other examples of measures for which this type of bounds hold. 
It was conjectured in \cite[Conjecture 5.1]{Lat} that the S-inequality holds for rotation-invariant measures whose densities are nonincreasing on half lines through the origin. This has been verified in dimensions lower than or equal to 3 (see \cite{SZ}).

Besides the Gaussian measure, S-inequality is known to hold for the exponential measure. Loosely speaking, in \cite{NT} it has been obtained as a by-product of the proof that for the measure $\dd \nu^n(x) = (1/2^n)e^{-|x_1| - \ldots - |x_n|} \dd x$ in $\R^n$ among the sets which are unions of coordinate parallelepipeds, coordinate-wise symmetric (called \emph{ideals}, see Section \ref{sec:results} for a proper definition)
and which have a fixed measure, the strips have the minimal measure of dilation.
The aim of the present paper is to extend this result for the measures $\nu_p^n$ on $\R^n$ with densities
\begin{equation}\label{eq:defnu}
	\dd \nu_p^n(x) = (c_p/2)^ne^{-|x|_p^p} \dd x, \qquad x \in \R^n,
\end{equation}
where we denote $|(x_1, \ldots, x_n)|_p = (\sum |x_i|^p)^{1/p}$ and $c_p = 1/\Gamma(1+1/p)$ is a normalization constant.

Proving Lata\l a's conjecture or at least providing other examples of measures supporting S-inequality still seems to be a challenge. It is worth recalling that some work has been done in the complex case. One considers the Gaussian measure $\dd \mu_n(z) = \left[1/(2\pi)^n\right]e^{-|z_1|^2-\ldots-|z_n|^2} \prod_{j=1}^n\dd \textrm{Re}z_j\dd \textrm{Im}z_j$ on $\C^n$ and asks whether
\begin{equation}\label{eq:complexS}
\mu_n(K) = \mu_n(C) \quad \Longrightarrow \quad \mu_n(tK) \geq \mu_n(tP), t \geq 1,
\end{equation}
for all convex \emph{circled}  sets $K \subset \C^n$ and all cylinders $C = \{z \in \C^n, \ |z_1| \leq R\}$. ($K$ is \emph{circled} if $z \in K$ implies $e^{it}z \in K$ for any $z \in \C^n$ and real $t$.) This was conjectured by A. Pe\l czy\'nski and seems to be a natural complex counterpart of the Lata\l a-Oleszkiewicz S-inequality. Following the methods from \cite{LO1}, the author obtained a partial result which says that there exists a universal constant $c > 0.64$ such that \eqref{eq:complexS} holds for $t \in [1,t_0]$ with $\mu_n(t_0K) = c$ (see \cite{Tko}). Using the inductive argument from \cite{KS}, the authors showed \eqref{eq:complexS} for all $K$ which are \emph{Reinhardt complete}, i.e. they are \emph{circled} with respect to each coordinate individually --- along with each point $(z_1, \ldots, z_n)$ such a set contains all points $(w_1, \ldots, w_n)$ for which $|w_k| \leq |z_k|$, $k = 1, \ldots, n$ (this is in fact a complex notion of ideals). Surprisingly, that result implies the real case of the S-inequality for the exponential measures. See \cite{NT} for the details.

In Section \ref{sec:results} we present our main results. Section \ref{sec:proofs} is devoted to their proofs.

\section{Results}\label{sec:results}

We begin with a few definitions. For a Borel measure $\mu$ on $\R$ its product measure $\mu \otimes \ldots \otimes \mu = \mu^{\otimes n}$ is denoted by $\mu^n$. We say that such a product measure $\mu^n$ on $\R^n$ \emph{supports the S-inequality for a Borel set} $L \subset \R^n$ if for any its dilation $K = sL$, $s > 0$, and any strip $P = \{x \in \R^n, \ |x_1| \leq p\}$ we have
\begin{equation}\label{eq:sineqdef}
	\mu^n(K) = \mu^n(P) \quad \Longrightarrow \quad \mu^n(tK) \geq \mu^n(tP), \qquad \textrm{for $t \geq 1$}.
\end{equation}
If we assume that the function $\Psi(x) = \mu\left( [-x,x]\right)$ is invertible for $x \geq 0$, we can write (\ref{eq:sineqdef}) as
\begin{equation}\label{eq:sineqconc}
	\mu^n(tK) \geq \Psi\Big[ t\Psi^{-1}\big(\mu(K)\big)\Big], \qquad \textrm{for $t \geq 1$}.
\end{equation}
A set $K \subset \R^n$ is called an \emph{ideal} if along with any its point $x \in K$ it contains the cube $[-|x_1|,|x_1|]\times \ldots \times [-|x_n|, |x_n|]$. 

Now we are able to state the main result.
\begin{thm}\label{thm1}
Let $p \in (0,1]$. Then the measure $\nu_p^n$ defined in \eqref{eq:defnu} supports the S-inequality for all ideals in $\R^n$.
\end{thm}
Thanks to simple coordinate-wise transport of measure argument we obtain the following corollary.
\begin{cor}\label{cor:weibullgamma}
For $p \in (0,1]$ and $\alpha > 0$ introduce the measure $\mu_{p, \alpha}$ on $\R$ with density
\begin{equation}\label{eq:measuremu}
	\dd \mu_{p, \alpha}(x) = \alpha c_p |x|^{\alpha - 1}e^{-|x|^{\alpha p}} \dd x.
\end{equation}
Then the product measures $\mu_{p, \alpha}^n$ supports the S-inequality for all ideals in $\R^n$.
In particular, defining for $\alpha > 0$ and $q \geq 1$ on $\R$ the symmetric Weibull measure $\omega_\alpha$ with the parameter $\alpha$ and the symmetric Gamma measure $\lambda_q$ with the parameter $q$ by
\begin{align}
	\label{eq:defweibull}\dd \omega_\alpha(x) &= \frac{1}{2}\alpha|x|^{\alpha - 1}e^{-|x|^\alpha}\dd x, \\
	\label{eq:defgamma}\dd \lambda_q (x) &= \frac{1}{2\Gamma(q)}q|x|^{q-1}e^{-|x|}\dd x.
\end{align}
we obtain that the product measures $\omega_\alpha^n$ and $\lambda_q^n$ support the S-inequality for all ideals in $\R^n$.
\end{cor}
The measures $\omega_\alpha^n$  provide the examples of distributions supporting the S-inequality and having both log-concave and log-convex tails. Indeed, the tail function of the Weibull distribution is $\omega_p\left( \{|x| > t\} \right) = e^{-t^\alpha}$ which is log-convex for $\alpha \in (0,1)$ and log-concave for $\alpha \in (1,\infty)$.

The fact that a measure supports the S-inequality for all ideals yields also the comparison of moments (see \cite[Corollary 2]{NT}). Here, the relevant result reads as follows.
\begin{cor}\label{cor:moments}
	Let $\|\cdot\|$ be a norm on $\R^n$ which is unconditional, i.e.
	\[
		\|(\epsilon_1 x_1, \ldots, \epsilon_n x_n)\| = \|(x_1, \ldots, x_n)\|
	\]
	for any $x_j \in \R$ and $\epsilon_j \in \{-1, 1\}$. Suppose that a product Borel probability measure $\mu^n = \mu^{\otimes n}$ supports the S-inequality for all ideals in $\R^n$. Then for $p \geq q > 0$
	\begin{equation}\label{eq.expmoments}
		\left( \int_{\R^n} \|x\|^p \dd \mu^n(x) \right)^{1/p} \leq C_{p,q} \left( \int_{\R^n} \|x\| ^q \dd \mu^n(x) \right)^{1/q},
	\end{equation}
	where the constant 
	\[ C_{p,q} = \frac{\left( \int_\R |x|^p \dd \mu(x) \right)^{1/p}}{\left( \int_\R |x|^q \dd \mu(x) \right)^{1/q}} 
	\] 
	is the best possible. In particular, we might take $\mu = \nu_p, \omega_\alpha, \lambda_q$, for $p \in (0, 1]$, $\alpha > 0$, $q \geq 1$ (see \eqref{eq:defnu}, \eqref{eq:defweibull}, \eqref{eq:defgamma}).
\end{cor}

\section{Proofs}\label{sec:proofs}

\subsection{Proof of Theorem \ref{thm1}}

The theorem is trivial in one dimension. For higher dimensions the strategy of the proof is to reduce the problem to the two dimensional case where everything can be computed. This is done in the following proposition.
\begin{proposition} \label{prop1}
Let $\mu$ be a Borel probability measure on $\R$. Let $\mu^n = \mu^{\otimes n}$ be its product measure on $\R^n$. If $\mu^2$ supports S-inequality for all ideals on $\R^2$ then for any $n \geq 2$ the measure $\mu^n$ supports S-inequality for all ideals on $\R^n$.
\end{proposition}
\begin{proof}
We proceed by induction on $n$. Let us fix $n \geq 2$ and assume that $\mu^n$ supports S-inequality for all ideals in $\R^n$. We would like to show that $\mu^{n+1}$ supports S-inequality for all ideals in $\R^{n+1}$. To this end consider an ideal $K \subset \R^{n+1}$ and set $t \geq 1$. Thanks to Fubini's theorem
\[
	\mu^{n+1}(tK) = \int_\R \mu^n((tK)_x) \dd \mu (x) = \int_\R \mu^n(tK_{x/t}) \dd \mu (x),
\]
where $A_x = \{y \in \R^n, \ (y,x) \in A\}$ is a section of a set $A \subset \R^{n+1}$ at a level $x \in \R$. For a set $A$ let $P_A$ denote a strip with a width $w_A$ such that $\mu^n(A) = \mu^n(P_A)$. Since the section $K_{x/t}$ is an ideal in $\R^n$, by the induction hypothesis we obtain
\[
	\mu^{n+1}(tK) \geq \int_\R \mu^n\left(tP_{K_{x/t}}\right) \dd \mu(x) = \int_\R \mu\left([-tw_{K_{x/t}}, tw_{K_{x/t}}]\right) \dd \mu (x).
\]
For the simplicity denote the function $x \mapsto w_{K_x}$ by $f$. If we put $G_f \subset \R^2$ to be an ideal \emph{generated} by $f$, i.e. $G_f = \{(x,y) \in \R^2, \ |y| \leq f(x), x \in \R\}$, then its dilation $tG_f$ is generated by the function $x \mapsto tf(x/t)$. Therefore
\[
	\int_\R \mu\left([-tw_{K_{x/t}}, tw_{K_{x/t}}]\right) \dd \mu (x) = \mu^2(tG_f).
\]
Yet, $\mu^2(G_f) = \mu^{n+1}(K)$, so taking the strip $P = [-w,w] \times \R^n$ with the same measure as $K$ we see that the strip $[-w,w]\times \R$ has the same measure as $G_f$. Now the fact that $\mu^2$ supports S-inequality implies $\mu^2(tG_f) \geq \mu^2(t([-w,w]\times \R)) = \mu^{n+1}(tP)$. Thus we have shown that $\mu^{n+1}(tK) \geq \mu^{n+1}(tP)$, which completes the proof.
\end{proof}

Thus it suffices to show the theorem when $n = 2$. Notice that any ideal $K \subset \R^2$ can be described by a nonincreasing function $f:\R_+ \to \R_+$, namely
\[
	K = \left\{ (x,y) \in \R^2, \; |y| \leq f(|x|)   \right\}.
\]
Fix such a function and take a strip $P=\{ |x_1| \leq w \}$ such that $\nu^2_p(K)=\nu^2_p(P)$. To prove that $\nu^2_p$ supports the S-inequality for the ideal $K$ it is enough to show that (see \cite[Proposition 1]{NT}) 
\[
	\po{}{t} \nu^2_p(tK) \Big|_{t=1} \geq \po{}{t} \nu^2_p(tP) \Big|_{t=1}.
\]
Let
\[
	M_p(K) = \call{K}{}{(|x|^p + |y|^p)}{\nu^2_p(x,y)}.
\]
We have
\[
	\nu^2_p(tK) = \frac{c_p^2}{4} \call{tK}{}{e^{-(|x|^p+|y|^p)}}{x \dd y} = \frac{c_p^2}{4} \call{K}{}{t^2e^{-t^p(|x|^p+|y|^p)}}{x \dd y} ,
\]
hence
\[
	\po{}{t} \nu^2_p(tK) \Big|_{t=1} = 2 \nu^2_p(K) -   p M_p(K).
\]
Therefore we are to prove that $M_p(K)\leq M_p(P)$.
Define the functions $T:\R_+ \to [0,1]$, $S:\R_+ \to [0,1]$
\[
	T(u) = c_p \call{u}{\infty}{e^{-x^p}}{x}, \qquad S(u)= c_p \call{0}{u}{x^pe^{-x^p}}{x}
\]
and let $\mu_+$ be the probability measure with density $c_pe^{-x^p}$ on $\R_+$.
Note that
\[
	S(u) = c_p \frac{1}{p}\call{0}{u}{x (-e^{-x^p})'}{x} = -\frac{c_p}{p} u e^{-u^p} + \frac{1}{p}(1-T(u)).   
\]
Thus $S(\infty) = 1/p$. We have
\begin{align*}
	M_p(K) & =  c_p^2 \call{0}{\infty}{ \call{0}{f(x)}{  (x^p+y^p) e^{-x^p-y^p}}{y}}{x} \\
	& = c_p \call{0}{\infty}{x^p e^{-x^p}(1-T(f(x)))}{x} 
	 +c_p \call{0}{\infty}{S(f(x)) e^{-x^p}}{x} \\
	& = \frac{1}{p} -  \call{0}{\infty}{x^p T(f(x))}{\mu_+(x)} 
	 + \call{0}{\infty}{S(f(x)) }{\mu_+(x)}.  
\end{align*}
To compute $M_p(P)$, it is enough to take $f(x) = \infty$ for $x < w$ and $f(x) = 0$ for $x \geq w$ in the above computations, we obtain
\begin{align*}
	\call{P}{}{(|x|^p + |y|^p)}{\nu^2_p(x,y)} & = \frac{1}{p} - \left(\frac{1}{p}-S(w) \right) + \frac{1}{p} \left( 1-T(w) \right) \\
	& = \frac{1}{p} +S(w) - \frac{1}{p} T(w).
\end{align*}
Let $\Phi:[0,1] \to \R$, $\Phi = S \circ T^{-1}$ and $g:\R_+ \to [0,1]$, $g=T \circ f$. We would like to prove
\[
	\call{}{}{\Phi(g)}{\mu_+} - \call{0}{\infty}{x^p g(x)}{\mu_+(x)} \leq S(w) - \frac{1}{p} T(w).
\]
Observe that 
\begin{align*}
	\nu^2_p(K) &= c_p^2 \call{0}{\infty}{ \call{0}{f(x)}{e^{-y^p-x^p}}{y}  }{x} \\
	&= \call{0}{\infty}{\left( 1-T(f(x)) \right)}{\mu_+(x)} = 1- \call{}{}{g}{\mu_+}.
\end{align*}
Our assumption $\nu^2_p(K)=\nu^2_p(P)$ yields $\call{}{}{g}{\mu_+} = T(w)$.
Moreover,
\[
	S(w) = \Phi(T(w)) = \Phi \left( \call{}{}{g}{\mu_+} \right).
\]
Therefore our inequality can be expressed in the following form
\[
	\call{}{}{\Phi(g)}{\mu_+} -  \Phi \left( \call{}{}{g}{\mu_+} \right) \leq \call{0}{\infty}{ g(x) \left( x^p - \frac{1}{p} \right) }{\mu_+(x)} .
\]
Note that $g:\R_+ \to [0,1]$ is nondecreasing. Summing up, to establish Theorem \ref{thm1} it suffices to prove the following lemma.
\begin{lemma} \label{lem1}
Let $p\in (0,1]$ and let $\mu_+$ be a measure with density $c_pe^{-x^p}$ supported on $\R_+$. Then for all nondecreasing functions $g:\R_+ \to [0,1]$ we have
\begin{equation}\label{eq2}
	\call{}{}{\Phi(g)}{\mu_+} -  \Phi \left( \call{}{}{g}{\mu_+} \right) \leq \call{0}{\infty}{ g(x) \left( x^p - \frac{1}{p} \right) }{\mu_+(x)} .
\end{equation}
\end{lemma}
In order to prove Lemma \ref{lem1} we shall need a lemma due to R. Lata\l a and K. Oleszkiewicz (see \cite[Lemma 4]{LO2} or \cite[Theorem 1]{Wol}). For convenience let us recall this result. 
\begin{lemma}[Lata\l a--Oleszkiewicz]\label{lem:LO}
Let $(\Omega,\nu)$ be a probability space and suppose that $\Phi:[0,1] \to \R$ has strictly positive second derivative and $1/\Phi''$ is concave. For a nonnegative function $g:\Omega \to [0,1]$ define a functional
\begin{equation} \label{eq3}
	\Psi_\Phi(g) = \call{\Omega}{}{\Phi(g)}{\nu} - \Phi  \left(  \call{\Omega}{}{g}{\nu} \right).
\end{equation}
Then $\Psi_\Phi$ is convex, namely
\[
	\Psi_\Phi(\lambda f + (1-\lambda)g) \leq  \lambda\Psi_\Phi( f) + (1-\lambda) \Psi_\Phi( g) .
\]
\end{lemma}

Now we show that our function $\Phi=S \circ T^{-1}$ satisfies the assumptions of Lemma \ref{lem:LO}.

\begin{lemma} \label{lem2}
The function $\Phi=S \circ T^{-1}:[0,1] \to \R$ satisfies $\Phi''>0$ and $(1/ \Phi'')'' \leq 0$.
\end{lemma}
\begin{proof}
Let $T^{-1}=F$. Note that $F'=\frac{1}{T'(F)}=-\frac{1}{c_p}e^{F^p}$. We have
\[
	\Phi' = S'(F) F' = c_p F^p e^{-F^p} \left( -\frac{1}{c_p}e^{F^p} \right) = - F^p
\]
and
\[
	\Phi'' = -p F^{p-1} F' = \frac{p}{c_p} F^{p-1} e^{F^p} > 0.
\]
Moreover,
\begin{align*}
	(1 /\Phi'')' &= \frac{c_p}{p} \left( F^{1-p} e^{-F^p} \right)' \\
	&=  \frac{c_p}{p} \left( (1-p) F^{-p} -p F^{1-p} F^{p-1} \right) e^{-F^p} F' = 1 - \frac{1-p}{p} F^{-p} 
\end{align*}
and
\[
	(1 /\Phi'')'' = (1-p) F^{-p-1}F' = - \frac{1-p}{c_p} F^{-p-1}  e^{F^p} \leq 0.
\]
\end{proof}

\begin{rem*}\label{rem:generaltransports}
The reader might want to notice that the last inequality is the place where the proof of the theorem does not work for other values of $p$.
\end{rem*}

We are ready to give the proof of Lemma \ref{lem1}.
\begin{proof}[Proof of Lemma \ref{lem1}]
 Combining Lemmas \ref{lem:LO} and \ref{lem2} we see that the left hand side of (\ref{eq2}) is a convex functional of $g$. The right hand side is linear in $g$ and therefore we see that $\lambda g_1 + (1-\lambda) g_2$ satisfies (\ref{eq2}) for every $\lambda \in [0,1]$ whenever $g_1,g_2$ satisfy (\ref{eq2}). Due to an approximation argument it suffices to prove our inequality for nondecreasing right-continuous piecewise constant functions having finite number of values. Every such a function is a convex combination of a finite collection of functions of the form $g_a(x)=\1_{[a,\infty)}(x)$, where $a \in [0,\infty]$. Therefore it suffices to check (\ref{eq2}) for the functions $g_a$. Since $\Phi(0)=S(\infty)=1/p$ and $\Phi(1)=0$ we have 
\[
 	\call{}{}{\Phi(g_a)}{\mu_+} -  \Phi \left( \call{}{}{g_a}{\mu_+} \right) = \frac1p \left( 1-T(a) \right) - S(a)
\] 
and
\[
	\call{0}{\infty}{ g_a(x) \left( x^p - \frac{1}{p} \right) }{\mu_+(x)} = \frac1p - S(a) - \frac1p T(a),
\]
thus we have equality in (\ref{eq2}).
\end{proof}
The proof of Theorem \ref{thm1} is now complete.

\subsection{Proof of Corollary \ref{cor:weibullgamma}}

The idea behind Corollary \ref{cor:weibullgamma} is that once a measure supports the S-inequality for all ideals then so does its image under properly chosen transformation (cf. the proof of Theorem 2 in \cite{NT}). Fix $p \in (0,1]$ and $\alpha > 0$. Consider the mapping $\fun{F}{(\R_+)^n}{(\R_+)^n}$ given by the formula
\[
	F(x_1, \ldots, x_n) = (x_1^\alpha, \ldots, x_n^\alpha).
\]
We will use it to change the variables. So, take an ideal $K \subset \R^n$, the strip $P \subset \R^n$ such that $\nu_p^n(K) = \nu_p^n(P)$, and compute the measure of the dilation $tK$ for some $t \geq 1$
\begin{align*}
	\nu_p^n(tK) &= \left( \frac{c_p}{2} \right)^n\int_{tK} e^{-|x|^p_p} \dd x = c_p^n\int_{tK \cap (\R_+)^n} e^{-\sum x_i^p} \dd x \\
	&= (\alpha c_p)^n\int_{F^{-1}(tK \cap (\R_+)^n)} \prod y_i^{\alpha - 1}e^{-y_i^{\alpha p}} \dd y.
\end{align*}
In the first equality we have used the symmetries of ideals, while in the last one we have changed the variables putting $x = F(y)$. Introducing the measure $\mu_{p, \alpha}$ on $\R$ with density \eqref{eq:measuremu}
we thus have seen that
\[
	\nu_p^n(tK) = \mu_{p, \alpha}(\widetilde{tK}),
\]
where for an ideal $A$ in $\R^n$ the set $\widetilde{A}$ denotes an ideal such that $\widetilde{A} \cap (\R_+)^n = F^{-1}(A \cap (\R_+)^n)$ (note that it makes sense as $F$ is monotone with respect to each coordinate). The point is that due to the homogeneity of $F$ we have $\widetilde{tK} = t^{1/\alpha}\widetilde{K}$. 
Moreover, strips are mapped onto strips. Therefore
\[
	\mu_{p, \alpha}(t^{1/\alpha} \widetilde{K}) = \nu_p^n(tK) \geq \nu_p^n(tP) = \mu_{p, \alpha}(t^{1/\alpha} \widetilde{P}),
\]
which means that $\mu_{p, \alpha}$ supports the S-inequality for the ideal $\widetilde{K}$. Since the ideal $K$ is arbitrary, we conclude that $\mu_{p, \alpha}$ supports the S-inequality for all ideals. To finish the proof notice that we recover Weibull and Gamma distribution setting respectively $p=1$, $\alpha = 1/p = q$, i.e. $\omega_\alpha = \mu_{1, \alpha}$, $\lambda_q = \mu_{1/q, q}$.

\begin{rem*}\label{rem:generaltransports}
We might use more general change of variables $y_i = V(x_i)$ for some increasing function $\fun{V}{\R_+}{\R_+}$, $V(0) = 0$ and ask whether we will derive the S-inequality for other measures than $\mu_{p, \alpha}$ exploiting the above technique. Since we would like to have $\widetilde{tK} = u(t)\widetilde{K}$ for a monotone function $u$, we check it would imply that $V(st) = CV(s)V(t)$, and $C$ is a constant. So $V$ should be a power function yet this case has been studied in the above proof. 
\end{rem*}

\section*{Acknowledgements}

We would like to thank Prof. Krzysztof Oleszkiewicz for his useful comments.

\noindent Piotr Nayar \\
\noindent Institute of Mathematics, University of Warsaw, \\
\noindent Banacha 2, \\
\noindent 02-097 Warszawa, Poland. \\
\noindent \texttt{nayar@mimuw.edu.pl}

\vspace{1em}

\noindent Tomasz Tkocz \\
\noindent Institute of Mathematics, University of Warsaw, \\
\noindent Banacha 2, \\
\noindent 02-097 Warszawa, Poland. \\
\noindent \texttt{tkocz@mimuw.edu.pl}

\end{document}